\documentclass[12pt]{article}
\usepackage{fullpage}
\usepackage{epsf,epsfig,amsfonts,amsgen,amsmath,amstext,amsbsy,amsopn,amsthm,amssymb, epstopdf,tikz,mathrsfs}
\usepackage{ebezier,eepic,mathtools,dsfont}
\usepackage{color,colordvi}
\usepackage{multirow}
\usepackage{url}
\usepackage{graphicx}
\usepackage{epsf}
\usepackage[format=hang, margin=10pt]{caption}
\usetikzlibrary{calc}
\newtheorem{theorem}{Theorem}[section]
\newtheorem{proposition}[theorem]{Proposition}

\newtheorem{conjecture}[theorem]{Conjecture}
\newtheorem{lemma}[theorem]{Lemma}
\newtheorem{fact}{Fact}

\theoremstyle{definition}

\newtheorem{claim}{Claim}

\begin{document}

\title{Some results on the maximal chromatic polynomials of $2$-connected $k$-chromatic graphs}
\author{Yan Yang\thanks{School of Mathematics, Tianjin University, Tianjin, China: yanyang@tju.edu.cn. Supported by NNSF of China under Grant 11971346. }}

\date{}

\maketitle

\begin{abstract}
In 2015, Brown and Erey conjectured that every $2$-connected graph $G$ on $n$ vertices with chromatic number $k\geq 4$ has at most $(x-1)_{k-1}\big((x-1)^{n-k+1}+(-1)^{n-k}\big)$ proper $x$-colorings for all $x\geq k$. Engbers, Erey, Fox, and He proved this conjecture for $x=k$. In this paper, we prove Brown and Erey's conjecture under the condition that either the clique number of $G$ is $k$, or the independent number of $G$ is $2$.
\medskip


\end{abstract}

\section{Introduction}

Let $k\in \mathbb{N}$ and $G=(V,E)$ be a simple graph. A {\it proper $k$-coloring} of $G$ is a mapping $c: V\rightarrow \{1,\ldots,k\}$ such that
$c(u)\neq c(v)$ whenever $uv\in E$. A graph is $k$-{\it colorable} if it has a proper $k$-coloring. The {\it chromatic number} $\chi(G)$ is the least $k$ such that $G$ is $k$-colorable. A graph $G$ with $\chi(G)=k$ is called $k$-{\it chromatic}.
%If $\chi(H)< \chi(G)=k$ for every proper subgraph $H$ of $G$, then $G$ is $k$-{\it critical}.
Two $k$-colorings $c$ and $c'$ are distinct if $c(v)\neq c'(v)$ for some vertex $v$ of $G$. Let $P(G,k)$ be the number of distinct proper $k$-colorings of $G$, it is a polynomial in $k$ and called {\it chromatic polynomial} of $G$. It was introduced by Birkhoff \cite{B1912} in 1912 with the hope of proving the Four Color Conjecture. For a book devoted to this topic, see \cite{DKT05}.

In this paper we focus on the maximum chromatic polynomials of $k$-chromatic graphs. We denote $(x)_{k}=x(x-1)\cdots(x-k+1)$ the $k$th {\it falling factorial} of $x$, denote $C_n$ and $P_n$ the cycle on $n$ vertices and the path on $n$ vertices respectively, $K_n$ the complete graph on $n$ vertices, and $K_{m,n}$ the complete bipartite graph whose two parts have $m,n$ vertices respectively. The $l$-core of a graph is the maximal subgraph with minimum degree at least $l$. Let $G_{n,k}$ be the graph on $n$ vertices obtained from a $k$-clique by adding an ear (with $n-k$ internal vertices, each of degree $2$) attached to two distinct vertices of the clique.

In 1971, Tomescu \cite{Tomescu71} conjectured that
$$P(G,k)\leq k!(k-1)^{n-k}$$
for every connected graph $G$ on $n$ vertices with $\chi(G)=k\geq 4$.
This conjecture had attracted wide attention since then, some partial results had been published, see \cite{BE15, Erey181, Erey18, KM20, KM19} for example; and  proved completely by Fox, He, and Manners \cite{FHM19} in 2019.

\begin{theorem}[\cite{FHM19}, Theorem 1] If $G$ is a connected graph on $n$ vertices with $\chi(G)=k\geq 4$, then
$$P(G,k)\leq k!(k-1)^{n-k},$$
with equality if and only if the $2$-core of $G$ is a $k$-clique.
\end{theorem}

In 1990, Tomescu \cite{Tomescu90} further conjectured the following generalization.

\begin{conjecture}[\cite{Tomescu90}]\label{con0} If $G$ is a connected graph on $n$ vertices with $\chi(G)=k\geq 4$, then for all $x\geq k$,
$$P(G,x)\leq (x)_k(x-1)^{n-k},$$ with equality if and only if the $2$-core of $G$ is a $k$-clique.
\end{conjecture}

Tomescu \cite{Tomescu90} proved Conjecture \ref{con0} for planar graphs with $k=4$. Brown and Erey \cite{BE15} proved Conjecture \ref{con0} for $x\geq n-2+\Big(\binom{n}{2}-\binom{k}{2}-n+k\Big)^2$ . Erey confirmed Conjecture \ref{con0} under two constraints that either the clique number of $G$ is $k$, or the independent number of $G$ is at most $2$ in  \cite{Erey18}; and Erey also reduced Conjecture \ref{con0} (for $k=4$) to a finite family of $4$-chromatic graphs in \cite{Erey181}. Engbers and Erey \cite {EE18} proved Conjecture \ref{con0} for $4$-chromatic claw-free graphs and for all $k$-chromatic line graphs. Knox and Mohar proved Conjecture \ref{con0} for $k=4$ and $k=5$ in \cite{KM20} and \cite{KM19} respectively. Long and Ren \cite{LR23} proved Conjecture \ref{con0} for some types of graphs, such as graphs with maximum degree $n-2$, etc. To author's knowledge, a complete proof of Conjecture \ref{con0} is still open.

In 2015, Brown and Erey \cite{BE15}, proposed a strengthening conjecture for $2$-connected graphs as follows.

\begin{conjecture}[\cite{BE15}]\label{con1} If $G$ is a $2$-connected $k$-chromatic graph of order $n>k\geq4$, then for all $x\geq k$,
$$P(G,x)\leq (x-1)_{k-1}\big((x-1)^{n-k+1}+(-1)^{n-k}\big),$$ with equality if and only if $G\cong G_{n,k}$.
\end{conjecture}

Engbers, Erey, Fox, and He \cite{EEFH21} proved Conjecture \ref{con1} for $x=k$. For convenience, we let $f_{n,k}(x)=(x-1)_{k-1}\big((x-1)^{n-k+1}+(-1)^{n-k}\big)$. And note that if $n=k$, then the $k$-chromatic graph is $K_k$, and $P(K_k, x)=(x)_k=f_{k,k}(x)$.

In this paper, we confirm Brown and Erey's conjecture under two addition constraints respectively, they are the clique number of $G$ is $k$, and the independent number of $G$ is $2$.

\section{Graphs with $\omega(G)=k$}

Let $\omega(G)$ be the clique number of a graph $G$. In this section, we proof that if $G$ contains some certain subgraphs then $P(G,x)<f_{n,k}(x)$, and then proof that Conjecture \ref{con1} holds when $\omega(G)=k$, i.e., $G$ has a subgraph $K_k$.

An {\it ear} of a graph $G$ is a maximal path whose internal vertices have degree $2$ in $G$. An {\it ear decomposition} of $G$ is a decomposition $Q_0,\ldots,Q_k$ such that $Q_0$ is a cycle and $Q_i$ for $i\geq 1$ is an ear of \ $Q_0\cup\cdots \cup Q_{i-1}$. It is well known that every $2$-connected graph has an ear decomposition.

\begin{theorem}[\cite{West01}, Theorem 4.2.8]\label{ear} A graph is $2$-connected if and only if it has an ear decomposition. Furthermore, every cycle in a $2$-connected graph is the initial cycle in some ear decomposition.
\end{theorem}

The following three known results concerning the chromatic polynomials of subgraphs and some special types of graphs will be used in our later proof.

\begin{proposition}[\cite{Erey181}, Proposition 2.1] \label{sub1} If $H$ is a connected subgraph of a connected graph $G$, then
$$P(G,x)\leq P(H,x)(x-1)^{|V(G)|-|V(H)|}.$$
\end{proposition}

\begin{theorem}[\cite{DKT05}, Theorem 1.3.2]\label{kglue} If $G$ and $H$ are two graphs and $G\cap H\cong K_r$, then
$$P(G\cup H, x)=\frac{P(G,x)P(H,x)}{P(K_r,x)}.$$\end{theorem}

\begin{lemma}[\cite{EG17}, Lemma 4.5]\label{path} The number of $x$-coloring of a path $P_t$ $(t\geq 4)$ is at most $\big((x-1)^2-1\big)(x-1)^{t-4}$ when we color the endpoints of the path with two fixed colors.
\end{lemma}

%By using them, we get the following Lemmas.

\begin{lemma}\label{oddc} Let $G$ be a $2$-connected $k$-chromatic graph of order $n>k\geq 4$. If $G$ contains a subgraph $H$ which is obtained from $K_k$, by attaching an ear $Q_1$ with even length to $K_k$, and $|V(H)|<|V(G)|$, then
$$P(G,x)< (x-1)_{k-1}\big((x-1)^{n-k+1}+(-1)^{n-k}\big).$$
\end{lemma}

\begin{proof} Suppose that $|V(Q_1)|=n_1$, and the endpoints of $Q_1$ are $u$ and $v$. Then $Q_1+uv$ is a cycle of odd length $n_1$, and $(Q_1+uv) \cap K_k=K_2$. From Theorem \ref{kglue} and Proposition \ref{sub1},
\begin{eqnarray*}P(G,x)&\leq &\frac{(x)_kP(C_{n_1},x)}{x(x-1)}(x-1)^{n-k-(n_1-2)}\\
&=&\frac{(x)_k}{x(x-1)}\big((x-1)^{n_1}-(x-1)\big)(x-1)^{n-k-n_1+2}\\
&=&\frac{(x)_k}{x(x-1)}\big((x-1)^{n-k+2}-(x-1)^{n-k-n_1+3}\big)\\
&< &f_{n,k}(x).
\end{eqnarray*}
The last inequality is strict because $|V(H)|<|V(G)|$, $n-k-n_1+3>1$.
\end{proof}

\begin{lemma}\label{2ear} Let $G$ be a $2$-connected $k$-chromatic graph of order $n>k\geq 4$. If $G$ contains a subgraph $H$  which is obtained from $K_k$, by first attaching an ear $Q_1$ whose endpoints are $u, v$ to $K_k$ , then attaching another ear $Q_2$ whose one endpoint in $V(Q_1)-\{u,v\}$ and another endpoint in $V(K_k)-\{u,v\}$ to $K_k\cup Q_1$ , and $|V(Q_1)|, |V(Q_2)|\geq 3$, then
$$P(G,x)< (x-1)_{k-1}\big((x-1)^{n-k+1}+(-1)^{n-k}\big).$$
\end{lemma}

\begin{proof} We suppose that $|V(Q_1)|=n_1$ and $|V(Q_2)|=n_2$, then $k+n_1-2+n_2-2\leq n$, i.e., $k+n_1+n_2\leq n+4$.

\noindent{\bf Case 1.} $n_2\geq 4$.

In this case, $3\leq n_1\leq n-k$. We first color $K_k\cup Q_1$, then $Q_2$. Because $(Q_1+uv)\cap K_k=K_2$ and $Q_2\cong P_{n_2}$, by Theorem \ref{kglue}, Lemma \ref{path} and Proposition \ref{sub1}, we have
\begin{eqnarray*}P(G,x)&\leq &\frac{(x)_kP(C_{n_1},x)}{x(x-1)}\big((x-1)^2-1\big)(x-1)^{n_2-4}(x-1)^{n-k-(n_1-2)-(n_2-2)}\\
&=&\frac{(x)_k}{x(x-1)}\big((x-1)^{n_1}+(-1)^{n_1}(x-1)\big)\big((x-1)^2-1\big)(x-1)^{n-k-n_1}\\
%&=&\frac{(x)_k}{x(x-1)}((x-1)^{n-k+2}-(x-1)^{n-k}+(-1)^{n_1}(x-1)^{n-k-n_1+3}+(-1)^{n_1+1}(x-1)^{n-k-n_1+1})\\
&\leq &\frac{(x)_k}{x(x-1)}\big((x-1)^{n-k+2}-(x-1)^{n-k}+(x-1)^{n-k-n_1+3}-(x-1)^{n-k-n_1+1}\big)\\
&=&\frac{(x)_k}{x(x-1)}\Big((x-1)^{n-k+2}-(x-1)^{n-k-n_1+1}\big(1-(x-1)^2+(x-1)^{n_1-1}\big)\Big)\\
&< &f_{n,k}(x).
\end{eqnarray*}
The last inequality is strict because $1-(x-1)^2+(x-1)^{n_1-1}\geq 1$ and $-(x-1)^{n-k-n_1+1}\leq-(x-1)$, both equalities hold if and only if $n-k=3$, i.e., $n_1=3$. But when $n_1=3$, the third from the last inequality is strict.

\noindent{\bf Case 2.} $n_2=3$.

We denote the ear $Q_1$ by $v_1v_2\cdots v_{n_1}$ in which $v_1=u$ and $v_{n_1}=v$ are endpoints of the ear.  We denote $Q_2$ by $v_ixy$ in which $2\leq i\leq n_1-1$ and $y\in V(K_k)-\{u,v\}$. We define $Q'_1=v_1\cdots v_i$ and $Q''_1=v_i\cdots v_{n_1}$, they are two subpaths of path $Q_1$ with length $i-1$ and $n_1-i$ respectively.

\noindent{\bf Subcase 2.1.} $n_1\geq 6$.

Because $n_1\geq 6$ and $2\leq i\leq n_1-1$, at least one of $Q'_1$ and $Q''_1$ has length not less than $3$. See Figure \ref{fig1} for example. Without loss of generality, we assume $Q''_1$ has length not less than $3$. Now we consider $H$ is obtained by first attaching the ear $Q'_1xy$ to $K_k$ whose endpoints are $u$ and  $y$, then attaching the ear $Q''_1$ to $K_k\cup Q'_1xy$ whose one endpoint in $V(Q'_1xy)-\{u,y\}$ and another endpoint in $V(K_k)-\{u,y\}$.  Because $|V(Q'_1xy)|\geq 3$ and $|V(Q''_1)|\geq 4$, we have $P(G,x)<f_{n,k}(x)$, from Case 1.

\begin{figure}[h!]
\begin{center}
\begin{tikzpicture}
[p/.style={circle,draw=black,fill=black,inner sep=1.4pt}]

\draw (0,0) circle (20pt);

\node (u) at (150:20pt) [p] {};
\node (y) at (90:20pt) [p] {};
\node (v) at (30:20pt) [p] {};

\node (x) at (90:40pt) [p] {};
\node (v1) at (140:50pt) [p] {};
\node (vi) at (110:60pt) [p] {};
\node (vi+1) at (70:60pt) [p] {};
\node (vn-1) at (40:50pt) [p] {};

\draw (0,0) node {$K_k$};
\draw (160:25pt) node {$u$};
\draw (115:65pt) node {$v_i$};
\draw (20:25pt) node {$v$};
\draw (82:43pt) node {$x$};
\draw (80:26pt) node {$y$};

\draw (140:63pt) node {$Q'_1$};
\draw (40:63pt) node {$Q''_1$};

\draw(u)--(v1);\draw(v1)--(vi);\draw(vi)--(vi+1);\draw(vi+1)--(vn-1);\draw(vn-1)--(v);\draw(x)--(y);\draw(vi)--(x);

\end{tikzpicture}
\caption{A subgraph $H$.}
\label{fig1}
\end{center}
\end{figure}
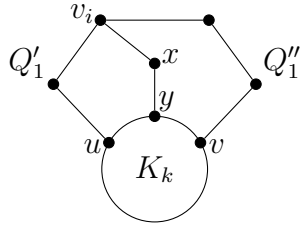

\noindent{\bf Subcase 2.2.} $n_1\leq 5$.

When $n_1=5$ or $3$, $Q_1$ has even length, from Lemma \ref{oddc}, we have $P(G,x)<f_{n,k}(x)$.
When $n_1=4$, $Q_1=v_1v_2v_3v_4$, $Q_2$ could be $v_2xy$ or $v_3xy$.  Without loss of generality, we assume $Q_2=v_2xy$.
Then $H-v_1v_2$ is a subgraph obtained by attaching the ear $yxv_2v_3v_4$ with length $4$ to $K_k$.
If $|V(G)|>|V(H-v_1v_2)|=V(H)$, then from Lemma \ref{oddc}, we have
$$P(G,x)\leq P(G-v_1v_2, x)<f_{n,k}(x).$$
If $|V(G)|=|V(H)|=k+3$, then $H$ is a spanning subgraph of $G$. Let $G'$ be the graph as shown in Figure \ref{fig2}, then $H=G'\cup K_k$ and $G'\cap K_k=K_3$.

\begin{figure}[h!]
\begin{center}
\begin{tikzpicture}
[p/.style={circle,draw=black,fill=black,inner sep=1.4pt}]
\node(v4) at(-1.2,-0.5)[p]{};
\node(v2) at(0,2)[p]{};
\node(y) at(1.2,-0.5)[p]{};
\node(v1) at(0,0.2)[p]{};
\node(v3) at ($(v2)!.5!(v4)$)[p] {};
\node(x) at ($(v2)!.5!(y)$)[p] {};

\draw (-1.5,-0.5) node {$v_4$}
      (1.5,-0.5) node {$y$}
      (0,2.3) node {$v_2$}
      (0,-0.1) node {$v_1$}
      (-0.8,1) node {$v_3$}
      (0.8,1) node {$x$};

\draw(v1)--(v2);\draw(v3)--(v4);\draw(v2)--(v3);\draw(v2)--(x);\draw(x)--(y);\draw(v1)--(y);\draw(v4)--(y);\draw(v4)--(v1);
\end{tikzpicture}
\caption{A subgraph of $G'$.}
\label{fig2}
\end{center}
\end{figure}
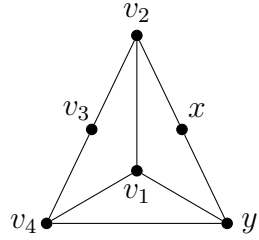
\noindent One can calculate that $$P(G',x)=x(x-1)(x-2)^2(x^2-3x+4).$$ From Proposition \ref{sub1} and Theorem \ref{kglue}, we have
\begin{eqnarray*}P(G,x)&\leq &P(H,x)\\
&=&\frac{(x)_kP(G',x)}{x(x-1)(x-2)}\\
&=&\frac{(x)_kx(x-1)(x-2)^2(x^2-3x+4)}{x(x-1)(x-2)}\\
%&=&\frac{(x)_k}{x(x-1)}(x(x-1)(x-2)(x^2-3x+4))\\
&=&\frac{(x)_k}{x(x-1)}((x-1)^5-(x-1)-x(x-1)(x-2)^2)\\
&<&f_{k+3,k}(x).
\end{eqnarray*}

Summarizing the two cases, the lemma follows.
\end{proof}

For the family of $2$-connected graph of order $n$, Tomescu obtained the following result on its maximum chromatic polynomial.

\begin{theorem}[\cite{Tomescu94}, Theorem 2.1]\label{2con} If $G$ is a $2$-connected graph of order $n\geq 3$, then for $x\geq 3$,
$$P(G,x)\leq (x-1)^{n}+(-1)^n(x-1),$$ with equality if and only if $G\cong C_n$ (or $G\cong K_{2,3}$ for the case that $n=5$ and $x=3$). \end{theorem}

Now we are ready to proof our one main result as follows.

\begin{theorem}\label{t1} If $G$ is a $2$-connected $k$-chromatic graph of order $n>k\geq 4$ and $\omega(G)=k$,  then $$P(G,x)\leq (x-1)_{k-1}\big((x-1)^{n-k+1}+(-1)^{n-k}\big),$$ with equality if and only if $G\cong G_{n,k}$.
\end{theorem}

\begin{proof} Because $G$ is $2$-connected and $\omega(G)=k$, from Theorem \ref{ear}, $G$ has a decomposition
$G_0, Q_1 \ldots, Q_t$ such that $G_0=K_k$ and $Q_i$ for $i\geq 1$ is an ear of $G_0\cup Q_{1}\cdots \cup Q_{i-1}$. (This decomposition can start from $K_k$ because we can choose a $k$-cycle from the $K_k$ arbitrarily, then all the other edges of $K_k$ are paths of length one with two endpoints are on the $k$-cycle.) We suppose that all ears $Q_i (1\leq i\leq t)$ have length at least $2$, because adding an ear of length $1$ will not increase the number of $x$-colorings of the graph. For $1\leq i\leq t$, we denote $|V(Q_i)|=n_i$, then we have $n_i\geq 3$, and $k+n_1+\cdots+n_t-2t=n$.

We define the set
\\$\Psi=\{Q_i: 2\leq i\leq t$, and one endpoint in the internal vertices of $Q_1$ and another endpoint \hspace*{2cm}in $V(G_0)-\{u,v\}$, in which $u,v$ are endpoints of $Q_1$\}.
\\If $\Psi\neq\emptyset$, then  we have $P(G,x)<f_{n,k}(x)$, from Lemma \ref{2ear}. We will discuss the case $\Psi=\emptyset$ in the following.

Suppose that $\mathcal{X}=\{Q_i: 1\leq i\leq t, ~\mbox{both endpoints in} ~V(G_0)\}$. Clearly, $\mathcal{X}\neq \emptyset$, because $Q_1\in \mathcal{X}$. For $Q_i\in \mathcal{X}$, we define the {\it block $B_i$ bounded by $Q_i$}, a $2$-connected subgraph of $G$ which is obtained in the following way.
\\{\bf Step 1.} From the cycle $Q_i+u_iv_i$ in which $u_i, v_i$ are endpoints of $Q_i$, we add all ears in $\{Q_1,\ldots,Q_t\}-\{Q_i\}$ whose two endpoints are in $V(Q_i+u_iv_i)$, then we obtain graph $B^{1}_i$, and we denote the set of ears we just added by $X^1_i$.
\\{\bf Step 2.} We add all ears in $\{Q_1,\ldots, Q_t\}-\{Q_i\}-X^1_i$ whose two endpoints are in $V(B^{1}_i)$, then we obtain graph $B^{2}_i$, and we denote the set of ears we just added by $X^2_i$.
\\{\bf Step 3.} We repeat this process until no more ear can be added to the current graph, then we get the block $B_i$ bounded by $Q_i$.

It is easy to see that each $B_i$ is $2$-connected, from Theorem \ref{ear}. We also note that if two ear in $\mathcal{X}$ share the same two endpoints, then we will get the same block bounded by these two ears. So for two different blocks, their boundaries share at most one common endpoint. And from the definition of a block and an ear, no two blocks share some common vertex
other than their possibly common endpoint of their boundaries.

\begin{claim}\label{c4} For a block $B$ whose boundary is $Q$,  if there is an ear $\tilde{Q}$ whose two endpoints $x,y$ are in $V(B)-\{u,v\}$, $V(G_0)-\{u,v\}$ respectively, in which $u,v$ are endpoints of $Q$, then  $P(G,x)<f_{n,k}(x)$.
\end{claim}

\begin{proof} When $x\in V(Q)$, from Lemma \ref{2ear}, $P(G,x)<f_{n,k}(x)$ follows. When $x\in V(B)-V(Q)$, then there exists an $\bar{x}x$-path $P_{\bar{x}x}$ such that $V(P_{\bar{x}x})\cap V(Q)=\{\bar{x}\}$, for $B$ is $2$-connected. If $\bar{x}$ is an internal vertex of $Q$, see Figure \ref{fig3}(a), then we get an ear $P_{\bar{x}x}\tilde{Q}$ whose two endpoints are in $V(Q)-\{u,v\}$ and $V(G_0)-\{u,v\}$ of $Q\cup K_k$. From Lemma \ref{2ear}, $P(G,x)<f_{n,k}(x)$ holds.
If there is no internal vertex $\bar{x}$ of $Q$ such that $V(P_{\bar{x}x})\cap V(Q)=\{\bar{x}\}$, then $\bar{x}=u$ (or $v$), and $x$ is on an ear of $K_k$ whose two endpoints are $u,v$, because $B$ is $2$-connected. We denote this ear $Q_{uxv}$, then $\tilde{Q}$ is an ear of $Q_{uvx}\cup K_k$. From Lemma \ref{2ear}, $P(G,x)<f_{n,k}(x)$ holds.
\end{proof}

\begin{figure}[h]
\begin{center}
\begin{tikzpicture}
[p/.style={circle,draw=black,fill=black,inner sep=1.4pt}]

\draw (0,0)  ellipse (30pt and 15pt);
\draw [rotate=285](-0.5,0) arc (40:320:1cm and 0.5cm);

\node (u) at (157:25pt) [p] {};
\node (y) at (100:15pt) [p] {};
\node (v) at (32:23pt) [p] {};
\node (x1) at (115:65pt) [p] {};
\node (x) at (115:45pt) [p] {};

\draw (0,0) node {$K_k$};
\draw (162:31pt) node {$u$};
\draw (115:72pt) node {$\bar{x}$};
\draw (30:30pt) node {$y$};
\draw (123:45pt) node {$x$};
\draw (83:19pt) node {$v$};
\draw (x)--(x1); \draw[-] (v) to [bend right=45](x);

\begin{scope}[xshift=6cm]

\draw (0,0)  ellipse (30pt and 15pt);
\draw [rotate=290](-0.55,0) arc (40:320:1cm and 0.5cm);
\draw [rotate=255](-0.5,0.8) arc (40:320:1cm and 0.5cm);

\node (a) at (158:25pt) [p] {};
\node (b) at (110:15pt) [p] {};
\node (c) at (65:16pt) [p] {};
\node (d) at (16:27pt) [p] {};
\node (x1) at (118:66pt) [p] {};
\node (y1) at (62:66pt) [p] {};
\node (x) at (120:45pt) [p] {};
\node (y) at (57:45pt) [p] {};
\draw (0,0) node {$K_k$};
\draw (x)--(x1) (y)--(y1); \draw[-] (y) to [bend right=30](x);

\draw (162:31pt) node {$a$}; \draw (128:45pt) node {$x$};\draw (115:22pt) node {$b$};
\draw (115:72pt) node {$\bar{x}$};\draw (49:45pt) node {$y$};\draw (60:72pt) node {$\bar{y}$};
\draw (60:22pt) node {$c$};\draw (20:35pt) node {$d$};
\end{scope}
\end{tikzpicture}
\\ (a)~~~~~~~~~~~~~~~~~~~~~~~~~~~~~~~~~~~~~~~(b)
\caption{Two subgraphs of  $G$.}
\label{fig3}
\end{center}
\end{figure}
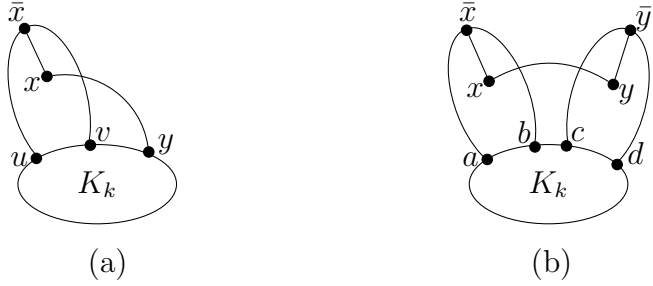

\begin{claim}\label{c3} For two blocks $B, \bar{B}$, if there is an ear $\hat{Q}$  whose two endpoints are in $V(B)-V(K_k)$, $V(\bar{B})-V(K_k)$ respectively, then $P(G,k)<f_{n,k}(x)$.
\end{claim}

\begin{proof} Suppose that the boundaries of $B, \bar{B}$ are $Q$ and $\bar{Q}$ respectively, in which the endpoints of $Q$ are $a,b$, the endpoints of $\bar{Q}$ are $c,d$, and it is possible $b=c$; $\bar{x}$ and $\bar{y}$ are one of the internal vertices of $Q$ and $\bar{Q}$ respectively.
If the  endpoints $x, y$ of $\hat{Q}$ are in $V(B)-V(Q)$, $V(\bar{B})-V(\bar{Q})$ respectively, as shown in Figure \ref{fig3} (b).
Because $\hat{Q}$ is an ear of $B\cup\bar{B}$ and there is a $yd$-path $P_{yd}$ in $\bar{B}$, we have a ear $\hat{Q}P_{yd}$ of $B\cup K_k$ whose two endpoints are in $V(B)-\{a,b\}$, $V(K_k)-\{a,b\}$ respectively. From Claim \ref{c4}, $P(G,k)<f_{n,k}(x)$. We notice that, when $\bar{x}=x$, or $\bar{y}=y$, or $\bar{x}=x$ and $\bar{y}=y$, $P(G,k)<f_{n,k}(x)$ still holds, from Lemma \ref{2ear}.
\end{proof}

\noindent{\bf Case 1.} $|V(B_1)|=n-k+2$.

In this case, $G_0\cup B_1=G$. %and $G_0\cap B_1=K_2$.
From Theorems \ref{kglue} and \ref{2con},

$$P(G,x)=\frac{P(G_0,x)P(B_1,x)}{x(x-1)}=\frac{(x)_kP(B_1,x)}{x(x-1)}\leq f_{n,k}(x).$$
The last equality holds when $B_1$ is the cycle $C_{n-k+2}$, that is  $G\cong G_{n,k}$.

\noindent{\bf Case 2.} $|V(B_1)|< n-k+2$.

In this case there exists some vertex not in $V(B_1)\cup V(G_0)$. And we need not consider the situations in the conditions of Claims \ref{c4} and \ref{c3}. So we only consider the graph $G$ which consists of $K_k$ and blocks.
%From Claims \ref{c4} and \ref{c3}, we need only consider the case
%for any pair of blocks $B, \bar{B}$, there is no ear whose two endpoints are in $V(B)$, $V(\bar{B})$ and there is no ear of any block.

We suppose that there are $m$ blocks $B_1,\ldots B_m$ with boundaries $Q_1,\ldots, Q_m$ , and $|V(B_i)|=b_i$ for $1\leq i\leq m$. We note that the following two inequalities hold, because the cycle maximal the number of $x$-colorings $(x\geq 4)$ in all $2$-connected graphs, from Theorem \ref{2con}.
\begin{equation}
P(B_i, x)\leq P(C_{n_i}, x),
\end{equation} with equality if and only if $B_i\cong C_{b_i}$. And
\begin{equation}
\frac{P(C_{n_i}, x)P(C_{n_j},x)}{x(x-1)}< P(C_{n_i+n_j-2}, x).
\end{equation}

By using inequalities $(1),(2)$ and $k+b_1+\cdots+b_m-2m=n$, we have
\begin{eqnarray*}P(G,x)&=&(x)_k\frac{P(B_1,x)}{x(x-1)}\cdots\frac{P(B_m,x)}{x(x-1)}\\
&\leq &(x)_k\frac{P(C_{b_1},x)}{x(x-1)}\cdots\frac{P(C_{b_m},x)}{x(x-1)}\\
&<&\frac{(x)_k}{x(x-1)}P(C_{n-k+2},x).
\end{eqnarray*}
The proof is completed.
\end{proof}

\section{Graphs with $\alpha(G)=2$}

In this section, we prove Conjecture \ref{con1} for graphs whose independence number is two. An {\it independent set} (or {\it stable set}) in a graph is a set of pairwise nonadjacent vertices.
The {\it independence number} of a graph $G$ is the maximum size of an independent set of vertices, denoted by $\alpha(G)$.
The connectivity of $G$ is the maximum value of $k$ for which $G$ is $k$-connected, denoted by $\kappa(G)$. If $S\subseteq V(G)$, we use $G[S]$ for the subgraph of $G$ induced by $S$. For any $u\in V(G)$, let $N_{G}(u)$ be the set of neighbors of $u$ in $G$ and $N_{G}[u]=N_{G}(u)\cup\{u\}$, $d_{G}(u)=|N_{G}(u)|$ be the degree of $u$ in $G$, and $\Delta (G)$ be the maximum degree of the graph $G$. For $u,v\in V(G)$, we denote $G/uv$ the graph obtained from $G$ by identifying $u$ and $v$ and replacing multiedges by single ones.

There are two well-known results on chromatic polynomials which will be used in this section, one is the recursive formula for computing $P(G,x)$, and the other is a formula for computing $P(G,x)$ when $\triangle (G)=|V(G)|-1$.

\begin{theorem}[\cite{DKT05}, Theorem 1.3.1]\label{RF} For a graph $G$ and $u,v\in V(G)$ with $uv\not\in E(G)$,
$$P(G,x)=P(G+uv, x)+P(G/uv, x).$$
\end{theorem}

\begin{lemma}[\cite{DKT05}, Corollary 1.5.1]\label{MF} For a graph $G$ and $u\in V(G)$, if $d_{G}(u)=|V(G)|-1$, then
$$P(G,x)=xP(G-u, x-1).$$
\end{lemma}

For distinct vertices $v_1,\ldots,v_i$ in $G$, let
$$G_{v_1,\ldots,v_i}=\left\{\begin{array}{cc}
              G,   &\mbox{if }  i=1, \\
              G+v_1v_i+\cdots+v_{i-1}v_i, &\mbox{if }  i\geq 2.
              \end{array}\right.
             $$
For a graph $G$ with $\triangle (G)<|V(G)|-1$, by using the recursive formula successively, one can deduce the following result.

\begin{lemma}[\cite{Dong00}, Lemma 3.3] \label{D00} For $u\in V(G)$ and $\{u_1,\ldots, u_t\}=V(G)\setminus N_{G}[u]$,
$$P(G,x)=P(G_{u_1,\ldots, u_t,u})+\sum\limits_{j=1}^{t}P(G_{u_1,\ldots,u_{j-1},u}/u_ju,x).$$
\end{lemma}

Next we give three inequalities which will be used in our later proof.

\begin{proposition} \label{ineqs} For $x,k\in \mathbb{N}$ and $x\geq 2$,
\\(1) $(x-2)_k\leq (x-1)^k-k(x-1)^{k-1},$
\\(2) $(x-2)^k\leq(x-1)^k-(x-1)^{k-1},$
\\(3) When $x,k \geq 2$, $(x-2)^k\leq (x-1)^k-k(x-1)^{k-1}+\frac{k(k-1)}{2}(x-1)^{k-2}, $ with equality if and only if $k=2$.
\end{proposition}

\begin{proof} The first inequality holds because
$$(x-2)_{k}=(x-1-1)(x-1-2)\cdots(x-1-k)\leq (x-1)^{k-1}(x-1-k).$$ In the same way, the second inequality holds.

For the third inequality, it is easy to check that when $k=2$, the equality is achieved. To complete the proof, we proof that
when $k\geq 3$ the inequality is strict. We proof this by induction on $k$.

When $k=3$, $(x-2)^3=x^3-6x^2+12x-8$ and $(x-1)^3-3(x-1)^2+3(x-1)=x^3-6x^2+12x-1$, so
$$(x-2)^3<(x-1)^3-3(x-1)^{3-1}+\frac{3\times 2}{2}(x-1)^{3-2}.$$

We assume that the third inequality is strict for $k=t\geq 4$. When $k=t+1$, we let
$$f(x)=(x-2)^{t+1}-(x-1)^{t+1}+(t+1)(x-1)^{t}-\frac{(t+1)t}{2}(x-1)^{t-1},$$ be a function of $x$, then its derivative
$$f'(x)=(t+1)\big((x-2)^{t}-(x-1)^{t}+t(x-1)^{t-1}-\frac{t(t-1)}{2}(x-1)^{t-2}\big),$$
by induction hypothesis, $f'(x)<0$, so $f(x)$ is strictly decreasing. And $f(2)=\frac{-(t-1)t}{2}<0$, so $f(x)<0$ holds when $x\geq 2$,
$(x-2)^k< (x-1)^k-k(x-1)^{k-1}+\frac{k(k-1)}{2}(x-1)^{k-2}$ follows when $k\geq 3$.
\end{proof}

For a connected graph $G$ with $\alpha(G)\leq 2$, Erey \cite{Erey18} obtained some structural results of $G$ and confirmed Conjecture \ref{con0} for these graphs.

\begin{lemma}[\cite{Erey18}, Proposition 2.1 and Lemma 2.3]\label{CS2} Let $G$  be a connected $k$-chromatic graph with $\alpha(G)=2$. If $G$ has a stable cut-set $S$ of size at most two, then
\\(i) $G\setminus S$ has exactly two connected components, say, $G_1$ and $G_2$,
\\(ii) $G_1$ and $G_2$ are complete graphs,
\\(iii) For every $u$ in $S$, either $V(G_1)\subseteq N_G(u)$ or $V(G_2)\subseteq N_G(u)$,
\\(iv) $\max\{\chi(G_1), \chi(G_2)\}\geq k-1$.
\\Furthermore, if $|S|=2$ and $\omega(G)< k$, then
\\(v) $G_1\cong K_{k-1}, G_2\cong K_{k-2}$,
\\(vi) every vertex in $S$ has at least one non-neighbor in $G_1$, and two vertices in $S$ have no common non-neighbor; all vertices in $S$ are adjacent to all vertices in $G_2$.
\end{lemma}

\begin{theorem}[\cite{Erey18}, Theorem 2.1]\label{In2} If $G$ is a connected graph with order $n$, $\chi(G)=k\geq 4$ and
$\alpha(G)\leq 2$, then
$$P(G,x)\leq (x)_k(x-1)^{n-k}.$$
Furthermore, the equality is achieved if and only if $G$ is a $k$-clique with a path of size one hanging off a vertex of the clique, or
$G$ is a $k$-clique with a path of size two hanging off a vertex of the clique, or $k=n$.
\end{theorem}

Note that when $G$ is a connected $k$-chromatic graph with $\alpha(G)=2$, if $S$ is a clique cut-set of $G$ with $|S|\leq 2$,
then items (i)-(iii) in Lemma \ref{CS2} still holds, from the proof in \cite{Erey18}; but items (v) and (vi) in Lemma \ref{CS2} will not happen,  due to the following Lemma.

\begin{lemma}\label{CS4} Let $G$ be a connected $k$-chromatic graph with $\alpha(G)=2$. If $G$ has a clique cut-set $S$ at most two, then $\omega(G)=k$.
\end{lemma}

\begin{proof} We denote by $G_1$ and $G_2$ the two connected components of $G\setminus S$ , denote by $G'_1$ and $G'_2$ the two subgraphs of $G$ induced by $V(G_1)\cup S$ and $V(G_2)\cup S$, respectively. Suppose that $|V(G_1)|=p$ and $|V(G_2)|=q$, then we have $G_1\cong K_p$ and $G_2\cong K_q$ from Lemma \ref{CS2}. Since $S$ is a clique, $\chi(G)=\max\{\chi(G'_1), \chi(G'_2)\}$. Without loss of generality, we assume $\chi(G)=\chi(G'_1)\geq\chi(G'_2)$.
Because $\chi(G)=k$, we have $\omega(G)\leq k$, $p, q\leq k$. To complete the proof, we proof $\omega(G)\geq k$.

{\bf Case 1.} $|S|=1$.

In this case, $|V(G'_1)|=p+1$ and $|V(G'_2)|=q+1$. Because $\chi(G'_1)=k$, we have $p+1\geq k$, i.e., $p\geq k-1$. So the value of $p$ can be $k-1$, or $k$. When $p=k$, then $G_1\cong K_k$, $\omega(G)\geq k$.
When $p=k-1$, we let $S=\{u\}$, if $V(G_1)\subseteq N_G(u)$, then $G'_1\cong K_k$, $\omega(G)\geq k$ follows. If $V(G_1)\not\subseteq N_G(u)$, then $\chi(G'_1)=k-1$ which contradicts $\chi(G'_1)=k$.

{\bf Case 2.} $|S|\neq 1$.

In this case, $G$ is $2$-connected, and $|S|=2$, $|V(G'_1)|=p+2$ and $|V(G'_2)|=q+2$. Because $\chi(G'_1)=k$, we have $p+2\geq k$, i.e., $p\geq k-2$. So the value of $p$ can be $k-2$, $k-1$, or $k$.

{\bf Subcase 2.1.} $p=k$.  In this subcase, $G_1\cong K_k$, $\omega(G)\geq k$ follows.

{\bf Subcase 2.2.} $p=k-1$.  In this subcase, we let $S=\{u,v\}$.

If $V(G_1)\subseteq N_{G}(u)$ or $V(G_1)\subseteq N_{G}(v)$, then we have a subgraph $K_k$ of $G$, so $\omega(G)\geq k$.

If $V(G_1)\not\subseteq N_{G}(u)$ and $V(G_1)\not\subseteq N_{G}(v)$, then $V(G_2)\subseteq N_{G}(u)$ and $V(G_2)\subseteq N_{G}(v)$ from item (iii) in Lemma \ref{CS2}, thus $G'_2\cong K_{q+2}$. If $q=k-2$, then $G'_2\cong K_{k}$, $\omega(G)\geq k$ follows.
If $q\leq k-3$, then $\chi(G'_2)=q+2\leq k-1$. If there are two distinct vertices $u'$ and $v'$ in $V(G_1)$ such that $uu', vv'\not\in E(G)$. Then we can find a proper $k-1$ coloring $c$ of $G'_1$ (first coloring $G_1$ with colors $1,\ldots, k-1$, then assign $c(u')$ and $c(v')$ to $u$ and $v$ respectively) which contradicts $\chi(G'_1)=k$. So there is only one vertex, say $w$ in $V(G_1)$ such that $uw, vw\not\in E(G)$. And all vertices
in $V(G_1)\setminus\{w\}$ are adjacent to both $u$ and $v$, then $\big(V(G_1)\setminus\{w\}\big)\cup S$  is a $k$-clique in $G$, $\omega(G)\geq k$ follows.

{\bf Subcase 2.3.} $p=k-2$. In this subcase, because  $|V(G'_1)|=k$ and $\chi(G'_1)=k$, we have $G'_1\cong K_k$, $\omega(G)\geq k$.

Summarizing the above, the lemma follows.
\end{proof}

\begin{lemma}\label{Min21} Let $G$ be a $2$-connected graph with order $n$, $\chi(G)=k\geq 4$ and $\alpha(G)=2$. If $G$ has a stable cut-set $S=\{u,v\}$ of size two, then $$P(G,x)\leq(x-1)_{k-1}\big((x-1)^{n-k+1}+(-1)^{n-k}\big),$$ with equality  if and only if $G\cong G_{k+1,k}$ or $G\cong G_{k+2,k}$.
\end{lemma}

\begin{proof} If $\omega(G)=k$, then $P(G,x)\leq f_{n,k}(x)$ with equality if and only if $G\cong G_{n,k}$ or $G\cong K_k$, from Theorem \ref{t1}.
 And $\alpha(G)=2$, so $P(G,x)=f_{n,k}(x)$ if and only if $G\cong G_{n,k}$ and $n\in \{k+1,k+2\}$, the lemma follows.

If $\omega(G)< k$, we denote $G_1$ and $G_2$ the two connected components of $G\setminus S$, then from Lemma \ref{CS2}, $G_1\cong K_{k-1}$, $G_2\cong K_{k-2}$, and $n=2k-1$. From Theorem \ref{RF},
\begin{equation}\label{F}
P(G,x)=P(G+uv, x)+P(G/uv, x).
\end{equation}

From Lemma \ref{CS2}, $u,v$ have no common non-neighbor in $G$ and they are adjacent to all vertices in $G_2$.
Then $G/uv$ can be seen as a vertex-gluing of $K_{k}$ and $K_{k-1}$. From Theorem \ref{kglue},
\begin{equation}\label{C}
P(G/uv, x)=\frac{P(K_{k},x)P(K_{k-1},x)}{x}=(x-1)_{k-1} (x)_{k-1}.
\end{equation}

For the graph $G+uv$, we let $H_1$ and $H_2$ be the subgraphs of $G+uv$ induced by the vertex sets $V(G_1)\cup S$ and  $V(G_2)\cup S$ respectively, then $G+uv$ can be seen as an edge-gluing of $H_{1}$ and $H_{2}$. From Theorem \ref{kglue},
\begin{equation*}
P(G+uv, x)=\frac{P(H_1,x)P(H_2,x)}{x(x-1)}.
\end{equation*}
 From Lemma \ref{CS2}, $H_2\cong K_k$. Since $G$ is $2$-connected, $H_1$ is $2$-connected, combining this with Lemma \ref{CS2}, $H_1$ contains a spanning subgraph which is isomorphic to the graph $G_{k+1,k-1}$. So $$P(H_1,x)\leq P(G_{k+1,k-1},x)=\frac{(x)_{k-1}\big((x-1)^4+(x-1)\big)}{x(x-1)}.$$ Then we have
\begin{eqnarray}\label{A}
P(G+uv, x)&\leq&\frac{(x)_{k-1}\big((x-1)^4+(x-1)\big)}{x(x-1)}\frac{(x)_k}{x(x-1)}\nonumber\\
&=&(x-1)_{k-2}\big((x-1)^3+1\big)(x-2)_{k-2}.
\end{eqnarray}

Combining \eqref{F}, \eqref{C} and \eqref{A}, we have
\begin{eqnarray*}P(G,x)%&=&P(G/uv,x)+P(G+uv, x)\\
%&\leq& (x-1)_{k-1}(x)_{k-1}+\frac{\frac{(x)_{k-1}((x-1)^4+(x-1))}{x(x-1)}(x)_k}{x(x-1)}\\
&\leq& (x-1)_{k-1}(x)_{k-1}+(x-1)_{k-2}\big((x-1)^3+1\big)(x-2)_{k-2}\\
&=&(x-1)_{k-1}(x-2)_{k-3}\big((x-1)^3+(x-1)^2+(x-1)+1\big)\\
&\leq&(x-1)_{k-1}\big((x-1)^{k-3}-(k-3)(x-1)^{k-4}\big)\big((x-1)^3+(x-1)^2+(x-1)+1\big)\\
&=&(x-1)_{k-1}\big((x-1)^k-(k-4)(x^2-x+1)(x-1)^{k-3}-(k-3)(x-1)^{k-4}\big),
\end{eqnarray*}
in which the second inequality holds from the first inequality in Proposition \ref{ineqs}.

When $k\geq 5$, we have $(k-4)(x^2-x+1)> 1$, then
$$(x-1)^k-(k-4)(x^2-x+1)(x-1)^{k-3}-(k-3)(x-1)^{k-4}<(x-1)^k-1,$$
and $$P(G,x)<(x-1)_{k-1}\big((x-1)^k-1\big)\leq (x-1)_{k-1}\big((x-1)^k+(-1)^{k-1}\big)=f_{2k-1,k}(x)$$ follows.

When $k=4$, from Lemma \ref{CS2}, $G$ has a spanning subgraph $G_0$ as shown in Figure \ref{fig4}. By calculating,
$$P(G_0,x)=(x-1)_3\big((x-1)^4-(x-1)(x^2-3x+1)-2\big).$$
Because $x\geq 4, x^2-3x+1>0$, we have $P(G,x)\leq P(G_0,x)< f_{7,4}(x)$.

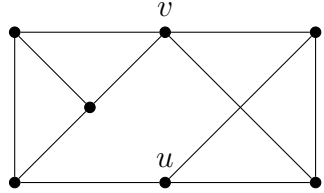
\begin{figure}[h]
\begin{center}
\begin{tikzpicture}
[p/.style={circle,draw=black,fill=black,inner sep=1.4pt}]
\node(11) at(0,0)[p]{};
\node(12) at(0,2)[p]{};
\node(13) at(1,1)[p]{};
\node(u) at(2,0)[p]{};
\node(v) at(2,2)[p]{};
\node(21) at(4,2)[p]{};
\node(22) at(4,0)[p]{};

\draw (2,2.3) node {$v$}
      (2,0.3) node {$u$};

\draw(11)--(12);\draw(13)--(11);\draw(12)--(13);\draw(v)--(12);\draw(u)--(11);\draw(13)--(v);\draw(v)--(21);\draw(v)--(22);
\draw(u)--(21);\draw(u)--(22);
\draw(21)--(22);
\end{tikzpicture}
\caption{The subgraph $G_0$ of $G$.}
\label{fig4}
\end{center}
\end{figure}

The proof is completed.

\end{proof}

\begin{theorem}\label{t2} If $G$ is a $2$-connected $k$-chromatic graph of order $n>k\geq 4$ and $\alpha(G)=\kappa(G)=2$, then $$P(G,x)\leq (x-1)_{k-1}\big((x-1)^{n-k+1}+(-1)^{n-k}\big),$$ with equality if and only if $G\cong G_{k+1,k}$ or $G\cong G_{k+2,k}$.
\end{theorem}

\begin{proof} Because $\kappa(G)=2$, $G$ has a cut-set of size two. Combining Theorem \ref{t1}, Lemmas \ref{CS4} and \ref{Min21}, the theorem follows.
\end{proof}

Due to Theorem \ref{t2}, we focus on $3$-connected graphs and a structural result is obtained.

\begin{lemma}\label{CS3} Let $G$  be a $3$-connected $k$-chromatic graph with $\alpha(G)=2$. If $G$ has a non-clique cut-set $S$ of size three, then
\\(i) $G\setminus S$ has exactly two connected components, say, $G_1$ and $G_2$,
\\(ii) $G_1$ and $G_2$ are complete graphs,
\\(iii) For every $u$ in $S$, either $V(G_1)\subseteq N_G(u)$ or $V(G_2)\subseteq N_G(u)$,
\\(iv) $\max\{\chi(G_1), \chi(G_2)\}\geq k-2$.
\\(v) if $\omega(G)< k$, then either $G_1\cong K_{k-1}, G_2\cong K_{k-3}$, or $G_1\cong G_2\cong K_{k-2}$,
\end{lemma}

\begin{proof} Items (i)-(iii) are easy to verified given that $\alpha(G)=2$. For item (iv), if both $\chi(G_1)$ and $\chi(G_2)$ are less than $k-2$, then we first properly color $G_1$ and $G_2$ by using $k-3$ colors, then color vertices in $S$ by using two colors not used on $G_1\cup G_2$, for $S$ is not a clique. Thus we get a proper $(k-1)$-coloring of $G$ which contradicts $\chi(G)=k$. Next we proof item (v).

Suppose that $G_1\cong K_p$ and $G_2\cong K_q$ in which $p\geq q$ and $S=\{x,y,z\}$. For $\alpha(G)=2$, $S$ is not a stable set.
Because $\omega(G)< k$ and $\chi(G_1)\geq k-2$, we have $p=k-1$ or $k-2$.

\noindent{\bf Case 1. $p=k-1$.}

In this case, every vertex in $S$ has at least one non-neighbour in $G_1$, otherwise $\omega(G)\geq k$. Furthermore we have the following Fact.

\begin{fact}There exist $x',y',z'\in V(G_1)$ and $x'\neq y'\neq z'$, such that $x',y',z'$ are non-neighbour of $x,y,z$ respectively.
\end{fact}

\begin{proof} For $\forall u,v\in S$, when $uv\in E(G)$, there exists $u',v'\in V(G_1)$ and $u'\neq v'$ such that $u',v'$ are non-neighbour of $u,v$ respectively, we note that $u',v'$ could be their common non-neighbours; otherwise, $u'=v'$ is the only non-neighbour of both $u$ and $v$ in $V(G_1)$, $G[(V(G_1)\setminus \{u'\})\cup S]\cong K_k$, thus $\omega(G)\geq k$. When $uv\not\in E(G)$, they have no common non-neighbours, because $\alpha(G)=2$. The Fact follows.
\end{proof}

From item (iii) in Lemma \ref{CS3}, for $\forall v\in S$, $V(G_2)\subseteq N_{G}(v)$. And $G[S]$ contains at least one edge, so we have $q\leq k-3$, otherwise $\omega (G)\geq k$. Next we prove $q=k-3$ by contradiction. If $q< k-3$, we first
properly color $G_1+S$ with $k-1$ colors in which $x,y,z$ has the same color with $x',y',z'$ respectively,
then we properly color $G_2$ with $k-1-3$ colors which are used on $V(G_1)-\{x',y',z'\}$, and we get a proper $(k-1)$-coloring of $G$ which contradicts $\chi(G)=k$. Hence, in this case $q=k-3$.

\noindent{\bf Case 2. $p=k-2$.}

In this case, for $\forall u,v\in S$ and $uv\in E(G)$, if $V(G_i)\subseteq N_{G}(u)$ (or $N_{G}(v)$), then $V(G_i)\not\subseteq N_{G}(v)$ (or $N_{G}(u)$), $i=1,2$, otherwise $\omega(G)\geq k$. Since $S$ is not a clique, the number of edges induced by $S$ could be one or two.  Without loss of generality, we suppose that the edge set induced by $S$ is $\{xy\}$ or $\{xy,xz\}$, we discuss those two subcases as follows.

\noindent{\bf Subcase 2.1. $E(G[S])=\{xy\}$.}

In this subcase, suppose without loss of generality that $V(G_1)\subseteq N_{G}(x)$, $V(G_2)\subseteq N_{G}(y)$, and $x',y'$ are non-neighbour of $x, y$ which are in $G_2$ and $G_1$ respectively. We prove $q=k-2$ by contradiction.
If $q<k-2$, we properly color $G_1$ by using $k-2$ colors, next color $x$ and $z$ by using a new common color, and color $y$ by using the same color with $y'$, thus we get a proper $(k-1)$-coloring of $G_1+S$. Then we can properly color $G_2$ by using $k-3$ colors that are used on $V(G_1)-\{y'\}$. Thus a proper $(k-1)$-coloring of $G$ is obtained which contradicts $\chi(G)=k$. Hence, in this subcase $q=k-2$.

\noindent{\bf Subcase 2.2. $E(G[S])=\{xy,xz\}$.}

In this subcase, suppose without loss of generality that $V(G_1)\subseteq N_{G}(x)$, $V(G_2)\subseteq N_{G}(y)$, $V(G_2)\subseteq N_{G}(z)$ and $x'$ is a non-neighbour of $x$ in $G_2$, $y'\neq z'$ are non-neighbours of $y, z$ in $G_1$ respectively. We prove $q=k-2$ by contradiction. If $q<k-2$,
firstly, we properly color $G_1$ by using $k-2$ colors, color $x$ by using a new color $a$, color $y, z$ by using the same color with $y', z'$ respectively, then we get a proper $(k-1)$-coloring of $G_1+S$. Then we color $x'$ in $G_2$ by color $a$, properly color $V(G_2)-\{x'\}$ by using
$k-4$ colors that used on $V(G_1)-\{y', z'\}$. Thus a proper $(k-1)$-coloring of $G$ is obtained which contradicts $\chi(G)=k$. Hence, in this subcase $q=k-2$.

Summarizing subcases 2.1 and 2.2, we get that $q=k-2$ in case 2. The proof is completed.
\end{proof}

\begin{lemma}\label{Min22} Let $G$  be a $3$-connected $k$-chromatic graph of order $n>k\geq 4$ and $\alpha(G)=2$. If $\omega(G)<k$ and $G$ has a non-clique cut-set $S$ of size three, then $$P(G,x)<(x-1)_{k-1}\big((x-1)^{n-k+1}+(-1)^{n-k}\big).$$
\end{lemma}

\begin{proof} Let $S=\{u,v,w\}$. Because $S$ is neither a clique set nor a stable set, we assume that $E_{G}(S)=\{uw\}$ or $E_{G}(S)=\{uw,vw\}$. From Lemma \ref{CS3}, we suppose that $G\setminus S$ has exactly two connected components $G_1$ and $G_2$ in which $G_1\cong K_{k-1}, G_2\cong K_{k-3}$, or $G_1\cong G_2\cong K_{k-2}$, $n=2k-1$.

We discuss these two cases in the following.

\noindent {\bf  Case 1. $E_{G}(S)=\{uw\}$.}

In this case, because both $u$ and $w$ are not adjacent to $v$, and $\alpha(G)=2$, all non-neighbors of $u$ or $w$ are adjacent to $v$.

\noindent {\bf  Subcase 1.1. $G_1\cong K_{k-1}, G_2\cong K_{k-3}$.}

In this subcase, every vertex in $S$ has at least one non-neighbor in $G_1$, because $\omega(G)< k$;
and every vertex in $S$ is adjacent to all vertices in $G_2$, from item (iii) in Lemma \ref{CS3}.
By using recursive formula twice, we get \begin{equation}\label{AC1}P(G,x)=P(G+vw+uv,x)+P(G/vw,x)+P\big((G+vw)/uv,x\big).\end{equation}

In the graph $G+vw+uv$, we denote the graph induced by $V(G_i)\cup S$ by $G'_i$ ($i=1,2$), then $G'_2\cong K_k$ and $G'_1$ has a subgraph
$G''_1\cong G_{k+1,k-1}$ in which $G''_1$ is obtained in the following way. Because $G$ is $3$-connected, there exist $\tilde{u},\tilde{v}\in V(G_1)$ and $\tilde{u}\neq \tilde{v}$, such that $\tilde{u}uv\tilde{v}$ is a 4-cycle in $G'_1$.  Gluing
$\tilde{u}\tilde{v}$ in $G_1$ with $\tilde{u}\tilde{v}$ in the $4$-cycle, we obtain $G''_1\cong G_{k+1,k-1}$. Then by gluing
$uv$ in $G''_1$ with $uv$ in $G'_2$, we obtain a subgraph $G'$ of $G+vw+uv$, see Figure \ref{fig7}(a). From Theorem \ref{kglue},
\begin{eqnarray}\label{AA}P(G+vw+uv)&\leq & P(G',x)\nonumber\\
&= &\frac{P(G''_1,x)P(G'_2,x)}{x(x-1)}\nonumber\\
&=&\frac{(x)_{k-1}\big((x-1)^4+x-1\big)}{x(x-1)}\frac{(x)_{k}}{x(x-1)}\nonumber\\
&=&(x-2)_{k-3}(x-1)_{k-1}\big((x-1)^3+1\big).
\end{eqnarray}

In the graph $G/vw$, we denote by $v_{w}$ the vertex obtained by identifying $v$ and $w$ in $G$. Then the graph induced by $V(G_2)\cup \{u, v_{w}\}$ is the $K_{k-1}$, and $v_{w}$ is adjacent to all vertices in $G_1$. We suppose $\tilde{u}\in V(G_1)$ is adjacent to $u$, then there is a subgraph $G_{k+1,k}$ which is obtained by gluing $\tilde{u}v_{w}$ in $G[V(G_1)\cup\{v_w\}]$ with $\tilde{u}v_{w}$ in the $3$-cycle $\tilde{u}uv_{w}$. When we glue edge $uv_{w}$ in $G_{k+1,k}$ with $\tilde{u}v_{w}$ in $G[V(G_2)\cup \{u, v_{w}\}]$, we obtain a subgraph $\hat{G}$ of $G/vw$, see Figure \ref{fig7}(b). From Theorem \ref{kglue},
\begin{eqnarray}\label{CC}P(G/vw,x)&\leq &P(\hat{G},x)\nonumber\\
&=& \frac{(x)_{k}\big((x-1)^3-(x-1)\big)}{x(x-1)}\frac{(x)_{k-1}}{x(x-1)}\nonumber\\
&=&(x-2)_{k-3}(x-1)_{k-1}\big((x-1)^2-1\big).
\end{eqnarray}

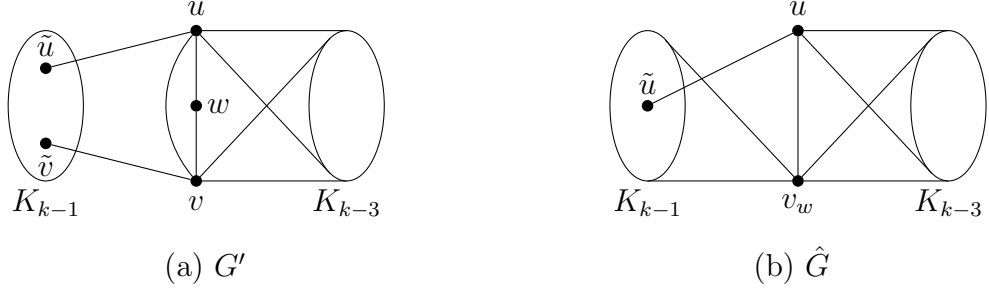
\begin{figure}[h]
\begin{center}
\begin{tikzpicture}
[p/.style={circle,draw=black,fill=black,inner sep=1.4pt}]

\node(v) at(2,0)[p]{};
\node(w) at(2,1)[p]{};
\node(u) at(2,2)[p]{};
\draw (2,2.3) node {$u$}
      (2.3,1) node {$w$}
      (2,-0.3) node {$v$};
\draw(u)--(w);\draw(v)--(w);\draw[-] (u) to [bend right=40](v);

\draw (0,1) ellipse [x radius=.5cm,y radius=1cm];
\node(v1) at(0,0.5)[p]{};
\node(u1) at(0,1.5)[p]{};
\draw (0,0.2) node {$\tilde{v}$}
      (0,1.8) node {$\tilde{u}$};
\draw(u)--(u1);\draw(v)--(v1);
\draw (0,-0.3) node {$K_{k-1}$};

\draw (4,1) ellipse [x radius=.5cm,y radius=1cm];
\draw(u)--(3.78,0.1);\draw(v)--(4,0);
\draw(u)--(4,2);\draw(v)--(3.78,1.9);
\draw (4,-0.3) node {$K_{k-3}$};

\begin{scope}[xshift=8cm]
\node(v) at(2,0)[p]{};
\node(u) at(2,2)[p]{};
\draw (2,2.3) node {$u$}
      (2,-0.3) node {$v_{w}$};
\draw(u)--(v);

\draw (0,1) ellipse [x radius=.5cm,y radius=1cm];
\node(u1) at(0,1)[p]{};
\draw (0,1.3) node {$\tilde{u}$};
\draw(u)--(u1);
\draw (0,-0.3) node {$K_{k-1}$};
\draw(v)--(0,0);\draw(v)--(0.24,1.87);

\draw (4,1) ellipse [x radius=.5cm,y radius=1cm];
\draw(u)--(3.78,0.1);\draw(v)--(4,0);
\draw(u)--(4,2);\draw(v)--(3.78,1.9);
\draw (4,-0.3) node {$K_{k-3}$};
\end{scope}
\end{tikzpicture}
\\(a) $G'$~~~~~~~~~~~~~~~~~~~~~~~~~~~~~~~~~~~~~~~~~~~~~~~~~(b) $\hat{G}$
\caption{The subgraph $G'$ of $G+vw+uv$ and the subgraph $\hat{G}$ of $G/vw$ in Subcase 1.1.}
\label{fig7}
\end{center}
\end{figure}

For graph $(G+vw)/uv$, with a similar argument to that for $G/vw$, we obtain that
\begin{equation}\label{CC1}P\big((G+vw)/uv,x\big)\leq (x-2)_{k-3}(x-1)_{k-1}\big((x-1)^2-1\big).\end{equation}

From \eqref{AC1},\eqref{AA},\eqref{CC} and \eqref{CC1}, we have that
\begin{equation}\label{R} P(G,x)\leq(x-1)_{k-1}(x-2)_{k-3}\Big((x-1)^3+1+2\big((x-1)^2-1\big)\Big)
\end{equation}

When $k\geq 5$, by using $(x-2)_{k-3}<(x-1)^{k-3}-(k-3)(x-1)^{k-4}$,
\begin{eqnarray*}&&(x-2)_{(x-3)}\big((x-1)^{3}+2(x-1)^2-1\big)\\
&<&(x-1)^k-(k-5)(x-1)^{k-1}-2(k-3)(x-1)^{k-2}-(x-1)^{k-3}+(k-3)(x-1)^{k-4}\\
&<&(x-1)^k-1,
\end{eqnarray*}
so we have
$$P(G,x)<(x-1)_{k-1}\big((x-1)^k-1\big)\leq f_{n,k}(x).$$

When $k=4$, up to isomorphism, there are three graphs $T_i$ $(1\leq i\leq 3)$ in this subcase which are shown in Figure \ref{fig5}.
By computing, we get that $$P(T_1,x)=(x-1)_3(x^4-5x^3+10x^2-8x),$$
$$P(T_2,x)=(x-1)_3(x^4-6x^3+14x^2-13x),$$ and $$P(T_3,x)=(x-1)_3(x^4-7x^3+19x^2-20x).$$
It is easy to check that $P(T_i,x)<f_{n,k}$, for $i=1,2,3.$

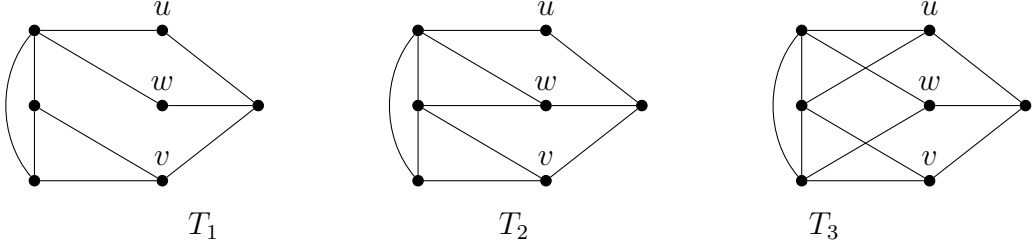
\begin{figure}[h!]
\begin{center}
\begin{tikzpicture}
[p/.style={circle,draw=black,fill=black,inner sep=1.4pt},xscale=0.85]
\node(1) at(0,2)[p]{};
\node(2) at(0,1)[p]{};
\node(3) at(0,0)[p]{};
\node(u) at(2,2)[p]{};
\node(w) at(2,1)[p]{};
\node(v) at(2,0)[p]{};
\node(4) at(3.5,1)[p]{};

\draw (2,2.3) node {$u$}
      (2,1.3) node {$w$}
      (2,0.3) node {$v$};

\draw(1)--(2);\draw(2)--(3);\draw(u)--(1);\draw(w)--(1);\draw(2)--(v);\draw(v)--(3);\draw(v)--(4);
\draw(u)--(4);\draw(w)--(4); \draw[-] (1) to [bend right=45](3);

\begin{scope}[xshift=6cm]
\node(1) at(0,2)[p]{};
\node(2) at(0,1)[p]{};
\node(3) at(0,0)[p]{};
\node(u) at(2,2)[p]{};
\node(w) at(2,1)[p]{};
\node(v) at(2,0)[p]{};
\node(4) at(3.5,1)[p]{};

\draw (2,2.3) node {$u$}
      (2,1.3) node {$w$}
      (2,0.3) node {$v$};

\draw(1)--(2);\draw(2)--(3);\draw(u)--(1);\draw(w)--(1);\draw(2)--(v);\draw(v)--(3);\draw(v)--(4);
\draw(u)--(4);\draw(w)--(4); \draw[-] (1) to [bend right=45](3);\draw(w)--(2);

\end{scope}

\begin{scope}[xshift=12cm]
\node(1) at(0,2)[p]{};
\node(2) at(0,1)[p]{};
\node(3) at(0,0)[p]{};
\node(u) at(2,2)[p]{};
\node(w) at(2,1)[p]{};
\node(v) at(2,0)[p]{};
\node(4) at(3.5,1)[p]{};

\draw (2,2.3) node {$u$}
      (2,1.3) node {$w$}
      (2,0.3) node {$v$};

\draw(1)--(2);\draw(2)--(3);\draw(u)--(1);\draw(w)--(1);\draw(2)--(v);\draw(v)--(3);\draw(v)--(4);
\draw(u)--(4);\draw(w)--(4); \draw[-] (1) to [bend right=45](3);\draw(u)--(2);\draw(w)--(3);

\end{scope}
\end{tikzpicture}
\\$T_1$~~~~~~~~~~~~~~~~~~~~~~~~~~~$T_2$~~~~~~~~~~~~~~~~~~~~~~~~~~~$T_3$
\caption{The graphs in Subcase 1.1 when $k=4$.}
\label{fig5}
\end{center}
\end{figure}

\noindent {\bf  Subcase 1.2 $G_1\cong G_2\cong K_{k-2}$.}

In this subcase, without loss of generality, we suppose that $V(G_1)\subseteq N_{G}(u)$, $V(G_2)\subseteq N_{G}(w)$, then $V(G_1)\subseteq N_{G}(v)$ or(and) $V(G_2)\subseteq N_{G}(v)$. For the symmetry, we assume that $V(G_1)\subseteq N_{G}(v)$.

In the graph $G+vw+uv$, we denote the graph induced by $V(G_i)\cup S$ by $G'_i$ ($i=1,2$), then $G'_1\cong K_k$ and $G'_2$ has a subgraph
$G''_2\cong G_{k+1,k-1}$ in which $G''_2$ is obtained in the following way.
Because $G$ is $3$-connected, there exist $\tilde{u},\tilde{v}\in V(G_2)$ and $\tilde{u}\neq \tilde{v}$, such that $\tilde{u}uv\tilde{v}$ is a 4-cycle in $G'_2$.  Gluing
$\tilde{u}\tilde{v}$ in $G_2$ with $\tilde{u}\tilde{v}$ in the $4$-cycle, we obtain $G''_2\cong G_{k+1,k-1}$. Then by gluing
$uv$ in $G''_2$ with $uv$ in $G'_1$, we obtain a subgraph $G'$ of $G+vw+uv$, see Figure  \ref{fig8}(a). From Theorem \ref{kglue},
\begin{eqnarray}\label{AAA}P(G+vw+uv)&\leq &P(G',x)\nonumber\\
&=&\frac{P(G'_1,x)P(G''_2,x)}{x(x-1)}\nonumber\\
&=&\frac{(x)_{k}}{x(x-1)}\frac{(x)_{k-1}\big((x-1)^4+x-1\big)}{x(x-1)}\nonumber\\
&=&(x-2)_{k-2}(x-1)_{k-2}\big((x-1)^3+1\big)\nonumber\\
&=&(x-2)_{k-3}(x-1)_{k-1}\big((x-1)^3+1\big).
\end{eqnarray}

In the graph $G/vw$, we denote by $v_{w}$ the vertex obtained by identifying $v$ and $w$ in $G$. Then the graph induced by $V(G_1)\cup\{u, v_{w}\}$ is the $K_{k}$. We suppose $\tilde{u}\in V(G_2)$ is adjacent to $u$ in $G_2$, then there is a subgraph $G_{k,k-1}$ which is obtained by gluing $\tilde{u}v_{w}$ in $G[V(G_2)\cup\{v_w\}]$ with $\tilde{u}v_{w}$ in the $3$-cycle $\tilde{u}uv_{w}$. When we glue edge $uv_{w}$ in $G_{k,k-1}$ with $\tilde{u}v_{w}$ in $K_{k}$, we obtain a subgraph $\hat{G}$ of $G/vw$, see Figure \ref{fig8}(b). From Theorem \ref{kglue},
\begin{eqnarray}\label{CCC}P(G/vw,x)&\leq& P(\hat{G},x)\nonumber\\
&=&\frac{(x)_{k-1}\big((x-1)^3-(x-1)\big)}{x(x-1)}\frac{(x)_{k}}{x(x-1)}\nonumber\\
&=&(x-2)_{k-3}(x-1)_{k-1}\big((x-1)^2-1\big).
\end{eqnarray}

\begin{figure}[h]
\begin{center}
\begin{tikzpicture}
[p/.style={circle,draw=black,fill=black,inner sep=1.4pt}]

\node(v) at(2,0)[p]{};
\node(w) at(2,1)[p]{};
\node(u) at(2,2)[p]{};
\draw (2,2.3) node {$u$}
      (2.3,1) node {$w$}
      (2,-0.3) node {$v$};
\draw(u)--(w);\draw(v)--(w);\draw[-] (u) to [bend right=40](v);

\draw (0,1) ellipse [x radius=.5cm,y radius=1cm];
\node(v1) at(4,0.5)[p]{};
\node(u1) at(4,1.5)[p]{};
\draw (4,0.2) node {$\tilde{v}$}
      (4,1.8) node {$\tilde{u}$};
\draw(u)--(u1);\draw(v)--(v1);
\draw(w)--(3.84,0.06);\draw(w)--(3.84,1.96);
\draw (0,-0.3) node {$K_{k-2}$};

\draw (4,1) ellipse [x radius=.5cm,y radius=1cm];
\draw(u)--(0.26,0.15);\draw(v)--(0,0);
\draw(u)--(0,2);\draw(v)--(0.24,1.88);
\draw (4,-0.3) node {$K_{k-2}$};

\begin{scope}[xshift=8cm]
\node(v) at(2,0)[p]{};
\node(u) at(2,2)[p]{};
\draw (2,2.3) node {$u$}
      (2,-0.3) node {$v_{w}$};
\draw(u)--(v);

\draw (0,1) ellipse [x radius=.5cm,y radius=1cm];
\node(u1) at(4,1)[p]{};
\draw (4,1.3) node {$\tilde{u}$};
\draw(u)--(u1);
\draw (0,-0.3) node {$K_{k-2}$};
%\draw(v)--(0,0);\draw(v)--(0.24,1.87);

\draw (4,1) ellipse [x radius=.5cm,y radius=1cm];
\draw(u)--(0.22,0.1);\draw(v)--(0,0);
\draw(u)--(0,2);\draw(v)--(0.22,1.9);
\draw (4,-0.3) node {$K_{k-2}$};
\draw(v)--(3.78,1.9);\draw(v)--(4,0);
\end{scope}
\end{tikzpicture}
\\(a) $G'$~~~~~~~~~~~~~~~~~~~~~~~~~~~~~~~~~~~~~~~~~~~~~~~~~(b) $\hat{G}$
\caption{The subgraph $G'$ of $G+vw+uv$ and the subgraph $\hat{G}$ of $G/vw$ in Subcase 1.2.}
\label{fig8}
\end{center}
\end{figure}
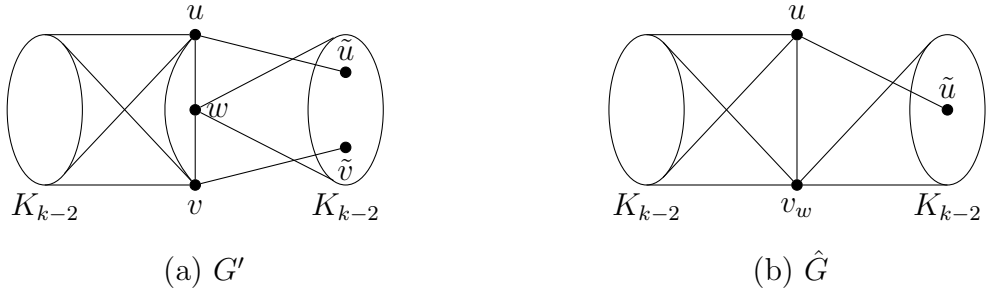

For graph $(G+vw)/uv$, with a similar argument to that for $G/vw$, we obtain that
\begin{eqnarray}\label{AACC}P((G+vw)/uv,x)\leq (x-2)_{k-3}(x-1)_{k-1}\big((x-1)^2-1\big).\end{eqnarray}

From \eqref{AC1},\eqref{AAA},\eqref{CCC} and \eqref{AACC}, we have that
\begin{eqnarray}\label{RR}P(G,x)\leq(x-1)_{k-1}(x-2)_{k-3}\Big((x-1)^3+1+2\big((x-1)^2-1\big)\Big).
\end{eqnarray}
The right side of inequality \eqref{RR} is exactly the same with that in inequality \eqref{R}. From the discussion in Subcase 1.1, we have $P(G,x)<f_{n,k}(x)$ when $k\geq 5$. When $k=4$, up to isomorphism, there is only one graph in this subcase which are shown in Figure \ref{fig6}. We find that this graph is isomorphism to $T_2$ in Figure \ref{fig5}, so result holds.

\begin{figure}[h!]
\begin{center}
\begin{tikzpicture}
[p/.style={circle,draw=black,fill=black,inner sep=1.4pt}]
\node(1) at(0,1.5)[p]{};
\node(2) at(0,0.5)[p]{};
\node(u) at(2,2)[p]{};
\node(w) at(2,1)[p]{};
\node(v) at(2,0)[p]{};
\node(3) at(4,1.5)[p]{};
\node(4) at(4,0.5)[p]{};

\draw (2,2.3) node {$u$}
      (2.2,1.3) node {$w$}
      (2,0.3) node {$v$};

\draw(1)--(2);\draw(4)--(3);\draw(u)--(w);\draw(u)--(1);\draw(2)--(u);\draw(w)--(1);\draw(v)--(1);
\draw(v)--(2); \draw(v)--(4);\draw(w)--(3);\draw(w)--(4);\draw(u)--(3);

\end{tikzpicture}
\caption{The graph in Subcase 1.2 when $k=4$.}
\label{fig6}
\end{center}
\end{figure}
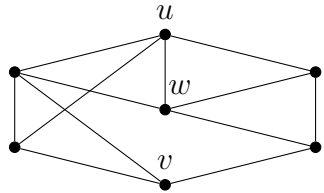

\noindent {\bf  Case 2. $E_{G}(S)=\{uw,vw\}$.}

In this case, by using recursive formula, we get $$P(G,x)=P(G+uv,x)+P(G/uv,x).$$  The discussion is similar to that in Case 1, and  also easier than that for Case 1, because we only use recursive formula once.

\noindent {\bf  Subcase 2.1. $G_1\cong K_{k-1}, G_2\cong K_{k-3}.$}

In this subcase,
\begin{eqnarray*}P(G,x)&\leq &\frac{(x)_{k}}{x(x-1)}\frac{(x)_{k-1}\big((x-1)^4+x-1\big)}{x(x-1)}+\frac{(x)_{k}}{x(x-1)}\frac{(x)_{k-1}\big((x-1)^3-(x-1)\big)}{x(x-1)}\\
&=&(x-2)_{k-2}(x-1)_{k-2}\big((x-1)^3+1\big)+(x-2)_{k-2}(x-1)_{k-2}\big((x-1)^2-1\big)\\
&=&(x-1)_{k-1}(x-2)_{k-3}\big((x-1)^3+(x-1)^2\big).
\end{eqnarray*}

\noindent {\bf  Subcase 2.2. $G_1\cong G_2\cong K_{k-2}$.}

In this subcase,
\begin{eqnarray*}P(G,x)&\leq &\frac{(x)_{k}}{x(x-1)}\frac{(x)_{k-1}\big((x-1)^4+x-1\big)}{x(x-1)}+\frac{(x)_{k}}{x(x-1)}\frac{(x)_{k-1}\big((x-1)^3-(x-1)\big)}{x(x-1)}\\
&=&(x-1)_{k-1}(x-2)_{k-3}\big((x-1)^3+(x-1)^2\big).
\end{eqnarray*}
When $k\geq 4$, by using $(x-2)_{k-3}<(x-1)^{k-3}-(k-3)(x-1)^{k-4}$, we have
\begin{eqnarray*}&&(x-1)_{k-1}(x-2)_{k-3}\big((x-1)^3+(x-1)^2\big)\\
&<&(x-1)_{k-1}\big((x-1)^k-(k-4)(x-1)^{k-1}-(k-3)(x-1)^{k-2}\big)\\
&<&f_{n,k}(x).
\end{eqnarray*}

Summarizing Cases 1 and 2, the Lemma is obtained.
\end{proof}

Now we give our another main result as follows.

\begin{theorem}\label{t3} If $G$ is a $2$-connected $k$-chromatic graph of order $n>k\geq 4$ and $\alpha(G)=2$, then $$P(G,x)\leq(x-1)_{k-1}\big((x-1)^{n-k+1}+(-1)^{n-k}\big),$$ with equality if and only if $G\cong G_{k+1,k}$ or $G\cong G_{k+2,k}$.
\end{theorem}

\begin{proof}  From Theorems \ref{t1} and \ref{t2}, we assume that $G$ is $3$-connected and $\omega(G)<k$. We prove this theorem by induction on the number of vertices. One can check that the result holds when $n=5$. We assume that the theorem holds for all $2$-connected graph $G$ with order less than $n>5$.

If $\triangle(G)=n-1$ and $d_{G}(u)=\triangle(G)$, then $\chi(G-u)=k-1$ and $\alpha(G-u)= 2$. By induction,
$$P(G-u, x)\leq f_{n-1,k-1}=(x-1)_{k-2}\big((x-1)^{n-k+1}+(-1)^{n-k}\big).$$
Because $d_{G}(u)=n-1$, form Lemma \ref{MF}, we have
\begin{eqnarray*}P(G,x)&=&xP(G-u,x-1)\\
&\leq& x(x-2)_{k-2}\big((x-2)^{n-k+1}+(-1)^{n-k}\big)\\
&\leq&(x-1+1)(x-2)_{k-2}\big((x-1)^{n-k+1}-(x-1)^{n-k}+(-1)^{n-k}\big)\\
&=&(x-1)_{k-1}\big((x-1)^{n-k+1}+(-1)^{n-k}\big)-(x-2)_{k-2}\big((x-1)^{n-k}+(-1)^{n-k-1}\big)\\
&<&(x-1)_{k-1}\big((x-1)^{n-k+1}+(-1)^{n-k}\big),
\end{eqnarray*}
in which the second inequality holds from the second inequality in Proposition \ref{ineqs}.

If $\triangle(G)<n-1$ and $d_{G}(u)=\triangle(G)$, we let $\{u_1,\ldots, u_t\}=V(G)\setminus N_{G}[u]$.
For $1\leq j\leq t$, if $G-\{u,u_{j}\}$ has a cut-vertex $w$ for some $j$, then $G$ has a cut set $\{u,u_j,w\}$. From Lemma \ref{Min22}, the theorem holds. Thus to complete the proof, we show that the theorem holds when $G-\{u,v_{j}\}$ is $2$-connected, for $1\leq j\leq t$ in the following.

By Lemma \ref{D00}, we have
$$P(G,x)=P(G_{u_1,\ldots, u_t,u},x)+\sum\limits_{j=1}^{t}P(G_{u_1,\ldots,u_{j-1},u}/u_ju,x).$$
For simplicity, we denote $G_{u_1,\ldots, u_t,u}$ by $G_0$, and denote $G_{u_1,\ldots,u_{j-1},u}/u_ju$ by $G_j$ for $1\leq j\leq t$, then
\begin{eqnarray}\label{s1} P(G,x)=P(G_0,x)+\sum\limits_{j=1}^{t}P(G_j,x).
\end{eqnarray}
Note that $G_0$ is obtained by joining each vertex in $V(G)\setminus N_{G}[u]$ to $u$, so $d_{G_0}(u)=n-1$ and $G_0-u\cong G-u$;
and for each $1\leq j\leq t$, $G_j\cong K_1\vee (G-\{u,u_j\})$, because $V(G)\setminus N_{G}[u]$ is a clique. Thus, by Lemma \ref{MF}, we obtain that
\begin{eqnarray}\label{s2} P(G_0,x)=xP(G_0-u,x-1)=xP(G-u,x-1),
\end{eqnarray}
and for $1\leq j\leq t$,
\begin{eqnarray}\label{s3} P(G_j,x)=xP(G-\{u,v_{j}\},x-1).
\end{eqnarray}
Note that $G-u$ is $2$-connected and $k-1\leq \chi(G-u)\leq k$; for each $1\leq j\leq t$, $G-\{u,u_{j}\}$ is $2$-connected and $k-1\leq \chi(G-\{u,u_{j}\})\leq k$. For their independence numbers, we have the following Fact.

\begin{fact}\label{d} For $1\leq j\leq t$, $\alpha(G-u)=\alpha(G-\{u,v_{j}\})=2$.\end{fact}
\begin{proof} Firstly we note that $\alpha(G-u)\leq \alpha(G)=2$ and $\alpha(G-\{u,u_{j}\})\leq\alpha(G)=2$. Then we proof equality holds by contradiction.

If $\alpha(G-u)=1$, then $G-u\cong K_{n-1}$, and $\chi(G)=\chi(G-u)=n-1$ which contradicts $\omega(G)<\chi(G)$. Hence $\alpha(G-u)=2$.

If there exists some $v_j$ such that $\alpha(G-\{u,v_{j}\})=1$,
then $G-\{u,v_{j}\}\cong K_{n-2}$. So we have $\chi(G)\geq\omega(G)\geq n-2$.
By Brook's Theorem, $\triangle(G)\geq \chi(G)\geq n-2$, so $\triangle(G)=\chi(G)=\omega(G)=n-2$ which contradicts $\omega(G)<\chi(G)$.
Hence $\alpha(G-\{u,v_{j}\})=2$ for each $1\leq j\leq t$.
\end{proof}

We also verify the following inequality holds
\begin{eqnarray}\label{s4}(x-1)_{k-1}\big((x-1)^{n-k}+(-1)^{n-k-1}\big)< (x-1)_{k-2}\big((x-1)^{n-k+1}+(-1)^{n-k}\big).\end{eqnarray}

Now we are ready to estimate $P(G-u,x)$ and $P(G-\{u,v_{j}\},x)$, $(1\leq j\leq t)$, by induction hypothesis and inequality \eqref{s4}, we obtain that
\begin{eqnarray}\label{s5} P(G-u,x)\leq f_{n-1,k-1}=(x-1)_{k-2}\big((x-1)^{n-k+1}+(-1)^{n-k}\big),
\end{eqnarray}
and for $1\leq j\leq t$,
\begin{eqnarray}\label{s6} P(G-\{u,v_{j}\},x)\leq f_{n-2,k-1}=(x-1)_{k-2}\big((x-1)^{n-k}+(-1)^{n-k-1}\big).
\end{eqnarray}

Combining \eqref{s1}, \eqref{s2}, \eqref{s3}, \eqref{s5} and \eqref{s6}, we have that
%\begin{eqnarray*}P(G_0,x)&=&xP(G_0-u,x-1)\\
%&\leq& x(x-2)_{k-2}\big((x-2)^{n-k+1}+(-1)^{n-k}\big),\end{eqnarray*}
%and for $1\leq j\leq t$,
%\begin{eqnarray*}P(G_j,x)&=&xP(G-\{u,v_{j}\},x-1)\\
%&\leq& x(x-2)_{k-2}\big((x-2)^{n-k}+(-1)^{n-k-1}\big).
%\end{eqnarray*}
%Then we have
$$P(G,x)\leq x(x-2)_{k-2}\big((x-2)^{n-k+1}+(-1)^{n-k}\big)+tx(x-2)_{k-2}\big((x-2)^{n-k}+(-1)^{n-k-1}\big).$$
By Brook's Theorem, $\triangle(G)\geq \chi(G)=k$, then we have $1\leq t=n-1-\triangle(G)\leq n-1-k$, and
%On the other hand, $t\leq k-1$,
%because $\omega(G)<k$ and all the non-neighbors of $u$ form a clique. Then we have $t\leq \min\{n-1-k, k-1\}$.
%For $\alpha(G)\chi(G)=2k\geq n$, we have $k-1\geq n-k-1$, $\min\{n-1-k, k-1\}=n-k-1$, then $t\leq n-k-1$.
\begin{eqnarray}\label{s7}P(G,x)&\leq &x(x-2)_{k-2}\big((x-2)^{n-k+1}+(-1)^{n-k}\big)\nonumber\\
&&+(n-k-1)x(x-2)_{k-2}\big((x-2)^{n-k}+(-1)^{n-k-1}\big)\nonumber\\
&=&x(x-2)_{k-2}\big((x-2)^{n-k}(x+n-k-3)+(-1)^{n-k-1}(n-k-2)\big).
\end{eqnarray}
We also note that $n\geq k+2$, because $t\geq 1$.
From the third equality in Proposition \ref{ineqs},
\begin{eqnarray}\label{s8}(x-2)^{n-k}\leq(x-1)^{n-k}-(n-k)(x-1)^{n-k-1}+\frac{(n-k)(n-k-1)}{2}(x-1)^{n-k-2}.\end{eqnarray}
For $\alpha(G)\chi(G)\geq n$, we have $2k\geq n$, and $x\geq k$, then we have
\begin{eqnarray}\label{s9} n-k-1\leq x-1. \end{eqnarray}

By \eqref{s8} and \eqref{s9} and $x=x-1+1$, one can expand and simplify inequality \eqref{s7}, we obtain
\begin{eqnarray*}P(G,x)&\leq &(x-2)_{k-2}\Big((x-1)^{n-k+2}-(x-1)^{n-k}-\frac{(n-k)(n-k-3)}{2}(x-1)^{n-k-1}\\
&&\qquad\qquad\qquad-\frac{(n-k)(n-k-1)}{2}(x-1)^{n-k-2}+(-1)^{n-k-1}(n-k-2)x\Big).
\end{eqnarray*}
If $n-k=2$, then we have
\begin{eqnarray*}P(G,x)&\leq&(x-2)_{k-2}\big((x-1)^4-(x-1)^2+(x-1)-1\big)\\
&<&(x-2)_{k-2}\big((x-1)^4-(x-1))\\
&<&f_{k+2,k}(x).
\end{eqnarray*}
the theorem holds.
\\If $n-k\geq 3$, it is easy to check that $-\frac{(n-k)(n-k-1)}{2}(x-1)^{n-k-2}+(-1)^{n-k-1}(n-k-2)x<0$, then we have
$$P(G,x)<(x-2)_{k-2}\big((x-1)^{n-k+2}-(x-1)\big)\leq f_{n,k}(x).$$ the theorem holds.

The proof is completed.
\end{proof}

\section{Remarks}

In 2021, Engbers, Erey, Fox, and He \cite{EEFH21} make a generation of Conjectures \ref{con0} and \ref{con1} to $l$-connected graphs $(l\geq 3)$ and prove the case $x=k$.

\begin{conjecture}[\cite{EEFH21}]\label{con2} Let $G$ be a $k$-chromatic $l$-connected graph on $n$ vertices with $k\geq 4$ and $l\geq 3$. Then for $x\geq k$
$$P(G,x)\leq (x)_k(x-1)^{n-l-k+1}+O\big((x-2)^n\big).$$
\end{conjecture}
All the Conjectures \ref{con0}, \ref{con1} and \ref{con2} are wild open.

\section*{Acknowledgement}
The author is grateful to Fengming Dong for reading this manuscript and helpful comments.

\bibliography{bibfile}
\end{document}